\documentclass[11pt]{amsart}
\usepackage{amsmath}
\usepackage{amssymb}
\usepackage{amsfonts}

\newtheorem{theorem}{Theorem}
\theoremstyle{plain}

\newtheorem{corollary}{Corollary}

\newtheorem{definition}{Definition}

\newtheorem{lemma}{Lemma}

\newtheorem{proposition}{Proposition}
\newtheorem{remark}{Remark}

\numberwithin{equation}{section}
\input{tcilatex}

\begin{document}
\title[The Spherical $\pi _{\alpha ,S^{n-1}}-$ Operator]{The Spherical $\pi
_{\alpha ,S^{n-1}}$-Operator }
\author{Dejenie A. Lakew}
\address{Virginia Union University\\
Department of Mathematics \\
Richmond, VA 23220}
\email{dalakew@vuu.edu}
\urladdr{http://www.vuu.edu}
\date{November 19, 2008}
\thanks{This paper is in final form and no version of it will be submitted
for publication elsewhere.}

\begin{abstract}
In this article we define the spherical $\pi _{\alpha ,S^{n-1}}$\ operator
over domains in the $\left( n-1\right) D-$ unit sphere $S^{n-1}$ of $%
%TCIMACRO{\U{211d} }%
%BeginExpansion
\mathbb{R}
%EndExpansion
^{n}$ \ and develop new and analogous results. We introduce a spherical
Dirac operator $\Gamma _{\alpha }:=\Gamma _{\omega }+\alpha $, where $\alpha
\in 
%TCIMACRO{\U{2102} }%
%BeginExpansion
\mathbb{C}
%EndExpansion
$ and $\Gamma _{\omega }=-\omega \wedge D_{\omega }$ , the anti-symmetric
Grassmanian product of $\omega $ with $D_{\omega
}=\dsum\limits_{i=1}^{n}e_{i}\frac{\partial }{\partial \omega _{i}}$. We use
a Gegenbauer polynomial $\Psi _{\alpha }^{n}(\omega -\upsilon )$ as a Cauchy
kernel for $\Gamma _{\alpha }$.
\end{abstract}

\maketitle

\section{\textbf{Introduction}}

The \ $\pi -$ operator is one of the tools used to study smoothness of
functions over Sobolev spaces and to solve some first order partial
differential equations such as the Beltrami equation. In Euclidean spaces,
we see that the singularity of its kernel is of order one more than the
dimension of the space $%
%TCIMACRO{\U{211d} }%
%BeginExpansion
\mathbb{R}
%EndExpansion
^{n}$ and hence it is a hyper singular integral operator.

\ 

In the class of singular integral operators, the $\pi $-operator is the
least studied integral operator than the weakly singular and singular
operators which are studied extensively.

\ \ \ 

Recently in \cite{lakryan} Dejenie A. Lakew and John Ryan also study the $%
\pi -$ operator in a generalized setting over Domain Manifolds in $%
%TCIMACRO{\U{2102} }%
%BeginExpansion
\mathbb{C}
%EndExpansion
^{n+1}$ and produced some properties and its integral representation as well.

\ \ 

In this paper we study the $\pi _{\alpha ,S^{n-1}}$-operator over domains in 
$S^{n-1}$, the $\left( n-1\right) D$ -unit sphere in $%
%TCIMACRO{\U{211d} }%
%BeginExpansion
\mathbb{R}
%EndExpansion
^{n}$. The differential operator we are considering is the spherical Dirac
operator 
\begin{equation*}
\Gamma _{\alpha }:=\Gamma _{\omega }+\alpha
\end{equation*}%
where 
\begin{equation*}
\Gamma _{\omega }=-\omega \wedge D_{\omega }
\end{equation*}%
and $\alpha $ is some complex number. Here%
\begin{equation*}
D_{\omega }=\dsum\limits_{i=1}^{n}e_{i}\frac{\partial }{\partial \omega _{i}}
\end{equation*}%
is the usual Dirac operator in $%
%TCIMACRO{\U{211d} }%
%BeginExpansion
\mathbb{R}
%EndExpansion
^{n}$ and $\wedge $ is the Grassman( or wedge) product. The function which
is used as a Cauchy kernel or fundamental solution to this spherical Dirac
operator is a Gegenbauer polynomial.

\ \ 

\section{\textbf{Preliminaries: Algebraic and Analytic}}

Let $e_{1}$,$e_{2},e_{3,}$...,$e_{n\text{ }}$ be orthonormal unit vectors
that generate $%
%TCIMACRO{\U{211d} }%
%BeginExpansion
\mathbb{R}
%EndExpansion
^{n}$. Then the anti-commutative algebra of dimension $2^{n}$ is the one
defined in terms of a negative inner product : 
\begin{equation*}
\langle x,y\rangle =-\dsum\limits_{i=1}^{n}x_{i}y_{i}.
\end{equation*}
Thus $\parallel x\parallel =-x^{2}$ . Under this structure we have : 
\begin{equation*}
e_{ij}+e_{ji}=-2\delta ij
\end{equation*}%
where $\delta _{ij}$ is the Kronecker delta. This algebra is called a
Clifford algebra and is denoted by $Cl_{n}$. Every element in this algebra
is represented by%
\begin{equation*}
x=\dsum\limits_{A}e_{A}x_{A}
\end{equation*}
where $e_{A}=e_{i_{1}i_{2}...i_{r}}$ for $A=\{i_{1}<i_{2}<...<i_{n}\}%
\subseteq \left\{ 1,2,3,...,n\right\} $ and $x_{A}\in 
%TCIMACRO{\U{211d} }%
%BeginExpansion
\mathbb{R}
%EndExpansion
$.

Thus by identifying \ the element $x=\left( x_{1},x_{2},...,x_{n}\right) $
of $%
%TCIMACRO{\U{211d} }%
%BeginExpansion
\mathbb{R}
%EndExpansion
^{n}$ with $\dsum\limits_{i=1}^{n}e_{i}x_{i}$ $\in Cl_{n}$, we imbed the
Euclidean space $\ $%
\begin{equation*}
%TCIMACRO{\U{211d} }%
%BeginExpansion
\mathbb{R}
%EndExpansion
^{n}\hookrightarrow Cl_{n}.
\end{equation*}
For $x,y\in Cl_{n}$ , their Clifford product $xy$ is written as a sum of
their inner product and their anti-symmetric Grassmanian product, as :%
\begin{equation*}
xy=x.y+x\wedge y.
\end{equation*}

In particular, for $x,y\in 
%TCIMACRO{\U{211d} }%
%BeginExpansion
\mathbb{R}
%EndExpansion
^{n},$ we have:

\begin{eqnarray*}
xy &=&\left( \dsum\limits_{i=1}^{n}e_{i}x_{i}\right) \left(
\dsum\limits_{j=1}^{n}e_{j}x_{j}\right) \\
&=&\dsum\limits_{i,j=1}^{n}e_{ij}x_{i}y_{j}
\end{eqnarray*}

\begin{eqnarray*}
&=&\dsum\limits_{i=j=1}^{n}e_{ii}x_{i}y_{i}+\dsum\limits_{i\neq
j}^{n}e_{ij}x_{i}y_{j} \\
&=&\underset{x.y}{\underbrace{-\dsum\limits_{i=1}^{n}x_{i}y_{i}}}-\underset{%
x\wedge y}{\underbrace{\dsum\limits_{i<j}^{n}e_{ij}\left(
x_{i}y_{j}-x_{j}y_{i}\right) }}
\end{eqnarray*}

\ \ \ \ \ \ \ \ \ \ \ \ \ \ \ \ \ \ \ \ \ \ \ \ \ \ \ \ \ \ \ \ \ \ \ \ \ \
\ \ \ \ \ \ \ \ \ \ \ \ \ \ \ \ \ \ \ \ \ \ \ \ \ \ \ \ \ \ \ \ \ \ \ \ \ \
\ \ \ \ \ \ \ \ \ \ \ \ \ \ \ \ \ \ \ \ \ \ \ \ \ \ \ \ \ \ \ \ \ \ \ \ \ \
\ \ \ \ \ \ \ \ \ \ \ \ \ \ \ \ \ \ \ \ \ \ \ \ \ \ 

Every non zero element of $%
%TCIMACRO{\U{211d} }%
%BeginExpansion
\mathbb{R}
%EndExpansion
^{n}$ is invertible : for $x\in 
%TCIMACRO{\U{211d} }%
%BeginExpansion
\mathbb{R}
%EndExpansion
^{n\ast }$, where $\ast $ indicates the tossing out of zero, \ 
\begin{equation*}
x^{-1}=\frac{-x}{\parallel x\parallel ^{2}}\ .
\end{equation*}

Also for every element 
\begin{equation*}
x=\dsum\limits_{A\subseteq \left\{ 1<2<...<n\right\} }e_{A}x_{A}
\end{equation*}%
of $Cl_{n}\left( 
%TCIMACRO{\U{211d} }%
%BeginExpansion
\mathbb{R}
%EndExpansion
\right) $, we define the Clifford conjugate $\overline{x}$ of $x$ by 
\begin{equation*}
\overline{x}:=\dsum\limits_{A\subseteq \left\{ 1<2<...<n\right\} }\overline{e%
}_{A}x_{A}
\end{equation*}%
where, for 
\begin{equation*}
e_{A}=e_{i_{1}}...e_{i_{k}}
\end{equation*}%
\begin{equation*}
\overline{e}_{A}=\overline{e}_{i_{k}}...\overline{e}_{i_{1}},\overline{e}%
_{j}=-e_{j},j=1,...,n,\overline{e}_{0}=e_{0}
\end{equation*}%
and therefore, we have a Clifford norm given by 
\begin{equation*}
\parallel x\parallel _{Cl}=\left[ x\overline{x}\right] _{0}.
\end{equation*}%
$\ \ \ $\ \ \ \ \ \ \ \ \ \ \ \ \ \ \ \ \ 

Thus, the unit sphere $S^{n-1}$ in $%
%TCIMACRO{\U{211d} }%
%BeginExpansion
\mathbb{R}
%EndExpansion
^{n}$ is described as : 
\begin{equation*}
S^{n-1}=\left\{ x\in 
%TCIMACRO{\U{211d} }%
%BeginExpansion
\mathbb{R}
%EndExpansion
^{n}:\parallel x\parallel _{Cl_{n}}=1\right\} .
\end{equation*}

Consider a $c^{1}-$ domain $\Omega $ $\subseteq S^{n-1}$ , a function $%
f:\Omega \rightarrow Cl_{n}$ has a representation given by : 
\begin{equation*}
f(x)=\dsum\limits_{A\subseteq \left\{ 1<2<...<n\right\} }e_{A}f_{A}(x)
\end{equation*}%
where $f_{A}:\Omega \rightarrow 
%TCIMACRO{\U{211d} }%
%BeginExpansion
\mathbb{R}
%EndExpansion
$. In this regard, a Clifford valued function over a domain is said to be $%
C^{k}$ if each component real valued function $f_{A}$ is $C^{k}$, and we say
such a function belongs to a Sobolev space $W^{p,k}\left( \Omega
,Cl_{n}\right) $ if each component function $f_{A}\in W^{p,k}\left( \Omega
,Cl_{n}\right) $.

\ \ 

Let $f\in c^{1}\left( \Omega ,Cl_{n}\right) \cap c\left( \overline{\Omega }%
,Cl_{n}\right) $ , $\alpha \in 
%TCIMACRO{\U{2102} }%
%BeginExpansion
\mathbb{C}
%EndExpansion
$ and $\omega \in S^{n-1}$ . Then

\begin{definition}
We define a spherical Dirac operator by $\ $%
\begin{equation*}
\Gamma _{\alpha }:=\Gamma _{\omega }+\alpha
\end{equation*}%
where 
\begin{equation*}
\Gamma _{\omega }=-\omega \wedge D_{\omega }
\end{equation*}

\begin{equation*}
=-\dsum\limits_{i<j}^{n}e_{ij}\left( \omega _{i}\frac{\partial }{\partial
\omega _{j}}-\omega _{j}\frac{\partial }{\partial \omega _{i}}\right)
\end{equation*}

and $D_{\omega }=\dsum\limits_{i=1}^{n}e_{i}\frac{\partial }{\omega _{i}}$
is the usual Dirac operator.
\end{definition}

\begin{definition}
A function $f\in C^{1}\left( \Omega \rightarrow Cl_{n}\right) $ is called a
spherical left monogenic function of order $\alpha $ if 
\begin{equation*}
\Gamma _{\alpha }f\left( x\right) =0,\forall x\in \Omega
\end{equation*}%
and is a spherical right monogenic function of order \ $\alpha $ if \ 
\begin{equation*}
f(x)\Gamma _{\alpha }=0,\forall x\in \Omega .
\end{equation*}
\end{definition}

\ \ \ \ \ \ 

Consider the generalized Gegenbauer function of degree $\alpha $ and of
order $\lambda :$ 
\begin{equation*}
C_{\alpha }^{\lambda }(z)=\frac{\Gamma (\alpha +2\lambda )}{\Gamma (\alpha
+1)\Gamma (2\lambda )}F\left( -\alpha ,\alpha +2\lambda ;\lambda +\frac{1}{2}%
;\frac{1}{2}\left( 1-z\right) \right)
\end{equation*}

where $F\left( a,b;c;d\right) $ is a hypergeometric function given by

$\ $%
\begin{equation*}
F\left( a,b;c;d\right) :=\dsum\limits_{k=1}^{\infty }\frac{\left( a\right)
_{k}\left( b\right) _{k}d^{k}}{\left( c\right) _{k}k!}
\end{equation*}%
for $\mid d\mid \leq 1$ with 
\begin{equation*}
\left( x\right) _{k}:=\frac{\Gamma \left( x+k\right) }{\Gamma \left(
x\right) }
\end{equation*}
which is simplified to $:\dprod\limits_{i=1}^{k}\left( x+i-1\right) $

\begin{proposition}
The Gegenbauer function with degree $\alpha $ and order $\lambda $ can be
re-written as:

\begin{equation*}
C_{\alpha }^{\lambda }(z)=\frac{\Gamma \left( \alpha +2\lambda \right) }{%
\Gamma \left( \alpha +1\right) \Gamma \left( 2\lambda \right) }%
\dsum\limits_{k=1}^{\infty }\left( \dprod\limits_{i=1}^{k}\left( \frac{%
-\alpha \left( \alpha +2\lambda \right) +\left( i-1\right) \left( 2\lambda
\right) +\left( i-1\right) ^{2}}{\lambda -\frac{1}{2}+i}\right) \right) 
\frac{\left( 1-z\right) ^{k}}{k!2^{k}}
\end{equation*}
\end{proposition}

\begin{proof}
The proof follows from the simplification of the right side of the
hypergeometric function 
\begin{equation*}
F\left( a,b;c;d\right) :=\dsum\limits_{k=1}^{\infty }\frac{\left( a\right)
_{k}\left( b\right) _{k}d^{k}}{\left( c\right) _{k}k!}
\end{equation*}
to the sum

\begin{equation*}
\dsum\limits_{k=1}^{\infty }\left( \dprod\limits_{i=1}^{k}\left( \frac{%
ab+\left( i-1\right) \left( a+b\right) +\left( i-1\right) ^{2}}{c+i-1}%
\right) \right) \frac{d^{k}}{k!}
\end{equation*}
and therefore the hypergeometric function with particular inputs $F\left(
-\alpha ,\alpha +2\lambda ;\lambda +\frac{1}{2};\frac{1}{2}\left( 1-z\right)
\right) $ is simplified to

\begin{equation*}
\dsum\limits_{k=1}^{\infty }\left( \dprod\limits_{i=1}^{k}\left( \frac{%
-\alpha \left( \alpha +2\lambda \right) +\left( i-1\right) \left( 2\lambda
\right) +\left( i-1\right) ^{2}}{\lambda -\frac{1}{2}+i}\right) \right) 
\frac{\left( 1-z\right) ^{k}}{k!2^{k}}.
\end{equation*}

\ \ 

Hence the\ Gegenbauer function is given by 
\begin{equation*}
C_{\alpha }^{\lambda }(z)=\frac{\Gamma \left( \alpha +2\lambda \right) }{%
\Gamma \left( \alpha +1\right) \Gamma \left( 2\lambda \right) }%
\dsum\limits_{k=1}^{\infty }\left( \dprod\limits_{i=1}^{k}\left( \frac{%
-\alpha \left( \alpha +2\lambda \right) +\left( i-1\right) \left( 2\lambda
\right) +\left( i-1\right) ^{2}}{\lambda -\frac{1}{2}+i}\right) \right) 
\frac{\left( 1-z\right) ^{k}}{k!2^{k}}.
\end{equation*}
\end{proof}

\bigskip\ \ \ \ 

From the above Gegenbauer function, a function is constructed to be a
fundamental solution (or Cauchy kernel) for the spherical Dirac operator $%
\Gamma _{\alpha }:=\Gamma _{\omega }+\alpha $ as

\begin{equation*}
\Psi _{\alpha }^{n}\left( \omega ,\upsilon \right) =\frac{\pi }{\sigma
_{n-1}\sin \pi \alpha }\left( C_{\alpha }^{\frac{n+1}{2}}\left( \omega
.\upsilon \right) -\omega \upsilon C_{\alpha -1}^{\frac{n+1}{2}}\left(
\omega .\upsilon \right) \right)
\end{equation*}%
See \cite{lanck},\cite{hlry} for details.

\ 

\begin{proposition}
The fundamental solution $\Psi _{\alpha }^{n}\left( \omega ,\upsilon \right) 
$ to the spherical Dirac operator can be written as:

$\Psi _{\alpha }^{n}\left( \omega ,\upsilon \right) =$

$\frac{\pi }{\sigma _{n-1}\sin \pi \alpha }\left( \frac{\Gamma \left( \alpha
+n+1\right) }{\Gamma \left( \alpha +1\right) \Gamma \left( n+1\right) }%
\dsum\limits_{k=1}^{\infty }\left[ 
\begin{array}{c}
\left( \dprod\limits_{i=1}^{k}\left( \frac{-\alpha \left( \alpha +n+1\right)
+\left( i-1\right) \left( n+1\right) +\left( i-1\right) ^{2}}{\frac{n}{2}+i}%
\right) \right) - \\ 
\omega v\frac{\Gamma \left( \alpha +n\right) }{\Gamma \left( \alpha \right)
\Gamma \left( n+1\right) }\dsum\limits_{k=1}^{\infty }\left(
\dprod\limits_{i=1}^{k}\left( \frac{\left( -\alpha +1\right) \left( \alpha
+n\right) +\left( i-1\right) \left( n+1\right) +\left( i-1\right) ^{2}}{%
\frac{n}{2}+i}\right) \right)%
\end{array}%
\right] \frac{\left( 1-\omega .v\right) ^{k}}{k!2^{k}}\right) $.
\end{proposition}

\bigskip\ \ \ \ \ 

Then using this as a Cauchy kernel , we define the following integral
transforms over function spaces which are $C^{1,\alpha },$ for $0\leq \alpha
<1$ or over Sobolev spaces $W^{p,k}\left( \Omega ,Cl_{n}\right) $ for $%
1<p<\infty $.

\bigskip\ \ \ \ \ \ 

Let $\Omega $ be a bounded smooth domain in $S^{n-1}$ and $f\in C^{1}\left(
\Omega ,Cl_{n}\right) $ , then we define, the Teodorescu transform as: 
\begin{equation*}
T_{\Omega }\left( f\right) (\upsilon )=\int_{\Omega }\Psi _{\alpha
}^{n}\left( \omega ,\upsilon \right) f(\omega )d\omega ,\text{ for }\upsilon
\in S^{n-1}
\end{equation*}%
which is the right inverse of the spherical Dirac operator $\Gamma _{\alpha
} $.

\bigskip\ \ \ \ 

Also we have a non-singular boundary integral operator given by \ 
\begin{equation*}
\digamma _{\partial \Omega }f\left( \upsilon \right) =\int_{\partial \Omega
}\Psi _{\alpha }^{n}\left( \omega ,\upsilon \right) n\left( \upsilon \right)
f(\upsilon )d\partial \Omega _{\omega },\text{for }\upsilon \notin \partial
\Omega
\end{equation*}%
An other boundary integral is the singular integral given by 
\begin{equation*}
\widetilde{F}_{\partial \Omega }f(v)=2\int_{\partial \Omega }\Psi _{\alpha
}^{n}\left( \omega ,\upsilon \right) n\left( \upsilon \right) f(v)d\partial
\Omega _{\omega },\text{ for }v\in \partial \Omega
\end{equation*}%
The last integral is seen in terms of the Cauchy principal value and is good
for computing non tangential limits of integrable functions on the boundary
and also for Plemelji formulae.

\ \ \ \ \ 

By arguments of continuity and denseness, the integral transforms can also
be extended over Sobolev spaces.

\ 

Also for $p\in \left( 1,\infty \right) $ and $k=0,1,2,...,$ the following
mapping properties hold: 
\begin{equation*}
T_{\Omega }:W^{p,k}\left( \Omega ,Cl_{n}\right) \rightarrow W^{p,k+1}\left(
\Omega ,Cl_{n}\right)
\end{equation*}%
and 
\begin{equation*}
F_{\partial \Omega }:W^{p,k-\frac{1}{p}}\left( \partial \Omega
,Cl_{n}\right) \rightarrow W^{p,k}\left( \Omega ,Cl_{n}\right)
\end{equation*}

\bigskip\ \ \ \ 

Note that the functions in $W^{p,k-\frac{1}{p}}\left( \partial \Omega
,Cl_{n}\right) $ are fractionally (or rationally) smooth and the $\digamma
_{\partial \Omega }$ is an operator which increases the smoothness of a
function in the Slobedeckij space $W^{p,k-\frac{1}{p}}\left( \partial \Omega
,Cl_{n}\right) $ by $\frac{1}{p}$, and hence it maps functions from
Slobedeckij spaces to Sobolev spaces.

\ \ 

That is, the boundary transform $\digamma _{\partial \Omega }$ retrieves
regularity(smoothness) exponents of functions in $W^{p,k}\left( \Omega
\right) $ which were lost by the trace operator as: 
\begin{equation*}
tr_{\partial \Omega }:W^{p,k}\left( \Omega ,Cl_{n}\right) \rightarrow W^{p,k-%
\frac{1}{p}}\left( \partial \Omega ,Cl_{n}\right) \text{ and }\digamma
_{\partial \Omega }f=\digamma _{\partial \Omega }\left( tr_{\partial \Omega
}f\right)
\end{equation*}

\bigskip\ \ \ 

In general, the function spaces $W^{p,\gamma }(\partial \Omega ,Cl_{n})$ ,
for $\gamma $ a fraction are called Slobedeckij spaces with the following
definition :

\begin{definition}
$f\in W^{p,\gamma }(\partial \Omega ,Cl_{n})$ if $\left( 1+\mid \xi \mid
^{\gamma }\right) \overset{\wedge }{f}\in L^{p}\left( \partial \Omega
,Cl_{n}\right) $
\end{definition}

where $\overset{\wedge }{f}$ is the Fourier transform of $f$\ and the norm
is therefore given by 
\begin{equation*}
\parallel f\parallel _{W^{p,\gamma }(\partial \Omega ,Cl_{n})}:=\left(
\dint\limits_{\partial \Omega }\left( 1+\mid \xi \mid ^{\gamma }\right)
^{p}\mid \widehat{f}\mid ^{p}d\partial \Omega \right) ^{\frac{1}{p}}
\end{equation*}%
and these function spaces are used as spaces of symbols of
pseudo-differential operators, in which, singular integral operators are
special types of pseudo-differential operators.

\ \ 

Symbols are strong tools to study boundedness of pseudo-differential
operators, where the symbol of a singular integral operator is bounded if
and only if the operator is bounded, see \cite{mikh}. In particular, it is
indicated in \cite{gurkah}, \cite{mikh} that the $\pi $-operator is bounded
by showing its symbol is bounded.

\ \ 

\begin{proposition}
Let $f\in BC^{1}\left( \Omega \rightarrow Cl_{n}\right) $, with a bounded
derivative. Then 
\begin{equation*}
\Gamma _{\alpha }T_{\Omega }f=f.
\end{equation*}
That is $T_{\Omega }$ is a right inverse of $\Gamma _{\alpha }$.
\end{proposition}

\begin{theorem}
(Borel-Pompeiu) For $f\in C^{1}\left( \Omega \rightarrow Cl_{n}\right) $, we
have 
\begin{equation*}
\chi _{\Omega }f=F_{\partial \Omega }f+T_{\Omega }\Gamma _{\alpha }f
\end{equation*}

where $\chi _{\Omega }$ is the usual characteristic function of the domain $%
\Omega $.
\end{theorem}

\begin{corollary}
(Cauchy Integral Formula for Spherical Monogenics)

\begin{equation*}
f\in \ker \Gamma _{\alpha }\Leftrightarrow \text{ \ }f(v)=F_{\partial \Omega
}f(v)
\end{equation*}
\end{corollary}

\begin{corollary}
From the Borel-Pompeiu and the CIFs, a traceless$\ \gamma -$regular function
is a null function over $\Omega $.
\end{corollary}

\section{\textbf{Fundamental Results on the Spherical Dirac Operator}}

In this section, we present fundamental results on $\Gamma _{\alpha }$ ,
solve boundary value problems over domains in the unit sphere like cases of
domains in Euclidean spaces, using the algebraic and analytic tools
presented in the preliminary.

\begin{proposition}
Let $g\in W^{2,1}\left( \Omega ,Cl_{n}\right) $, $h\in W^{2,\frac{1}{2}%
}\left( \partial \Omega ,Cl_{n}\right) $. Then the inhomogeneous BVP:

\begin{equation*}
\left\{ 
\begin{array}{c}
\Gamma _{\alpha }f=g\text{ on }\Omega \\ 
trf=h\text{, on }\partial \Omega%
\end{array}%
\right.
\end{equation*}

has a unique solution $f\in W^{2,2}\left( \Omega ,Cl_{n}\right) $ given by%
\begin{equation*}
f=F_{\partial \Omega }h+T_{\Omega }g.
\end{equation*}

which is almost a $C^{2}-$function for no $%
%TCIMACRO{\U{211d} }%
%BeginExpansion
\mathbb{R}
%EndExpansion
^{n}\supseteq \Omega $ but is almost a $C^{1}-$function in $%
%TCIMACRO{\U{211d} }%
%BeginExpansion
\mathbb{R}
%EndExpansion
^{1}$
\end{proposition}

\begin{proof}
The unique solution of the BVP is obtained using the Borel-Pompeiu formula
which is given by%
\begin{equation*}
f=F_{\partial \Omega }h+T_{\Omega }g.
\end{equation*}
\end{proof}

\begin{corollary}
The analytic solution $f$\ of the BVP given above is almost a $C^{1}-$
function in $%
%TCIMACRO{\U{211d} }%
%BeginExpansion
\mathbb{R}
%EndExpansion
^{1}$ but almost a $C^{2}-$ function for no $%
%TCIMACRO{\U{211d} }%
%BeginExpansion
\mathbb{R}
%EndExpansion
^{n}$.
\end{corollary}

\begin{proof}
The solution $f$ given above\ is a function in the Sobolev space $%
W^{2,2}\left( \Omega ,Cl_{n}\right) $ and is almost a $C^{k}$-function over $%
\Omega $ contained in $%
%TCIMACRO{\U{211d} }%
%BeginExpansion
\mathbb{R}
%EndExpansion
^{n}$, if 
\begin{equation*}
2>\frac{n}{2}+k
\end{equation*}%
where $k\in 
%TCIMACRO{\U{2115} }%
%BeginExpansion
\mathbb{N}
%EndExpansion
$.

But the last inequality holds only when $k=1$ and $n=1$ and that prove the
result.
\end{proof}

\begin{proposition}
$\left( \text{Representation of a Function with Compact Support}\right) $

Let $f\in C_{c}^{1}\left( \Omega \rightarrow Cl_{n}\right) $ . Then $f$ has
a representation given by 
\begin{equation*}
f(v)=T_{\Omega }\left( \Gamma _{\alpha }f\right) (v)
\end{equation*}%
for $\nu \in S^{n-1}$.
\end{proposition}

\begin{proof}
Let $f$ be a $C^{1}-$ function with compact support over $\Omega \subseteq
S^{n-1}$. Then from Borel-Pompeiu formula we have the required result, since
the boundary integral is zero.

\ \ \ 

We see here that $T_{\Omega }$ is both right and left inverse of the
spherical Dirac operator $\Gamma _{\alpha }$.
\end{proof}

\begin{proposition}
$\left( \text{Representation of a Global Function}\right) $

If $\Omega $ is a global domain in the unit sphere, then every function $f$
in $W^{2,2}(\Omega ,Cl_{n})$(or in $C^{1}(\Omega ,Cl_{n})$ ) can be
represented over $\Omega $ by%
\begin{equation*}
f\left( v\right) =\dint\limits_{\Omega }\Psi _{\alpha }^{n}\left( w-v\right)
\Gamma _{\alpha }\left( w\right) f\left( w\right) d\Omega _{w}.
\end{equation*}
\end{proposition}

\begin{proof}
$\Omega $ is a global domain in the unit sphere means that $\Omega $ is the
whole sphere. Thus as the sphere is a boundary hypersurface,

its boundary is empty set ( we recall from differential topology that $%
\partial \partial =\varnothing $). Therefore the $\partial -$integral of $f:$
\begin{eqnarray*}
\dint\limits_{\partial \Omega }\Psi _{\alpha }^{n}\left( w-v\right)
n(w)f\left( w\right) d\Omega _{w} &=&\dint\limits_{\partial \Omega =\partial
\partial \left( \bullet \right) }\Psi _{\alpha }^{n}\left( w-v\right)
n(w)f\left( w\right) d\Omega _{w} \\
&=&\ \dint\limits_{\varnothing }\Psi _{\alpha }^{n}\left( w-v\right)
n(w)f\left( w\right) d\Omega _{w} \\
&=&0
\end{eqnarray*}%
and therefore from Borel-Pompeiu formula, we have the result.
\end{proof}

\bigskip\ \ \ 

The Lebesgue space $L^{2}\left( \Omega ,Cl_{n}\right) $ with a Clifford
valued inner product given by

\begin{equation}
\left\langle f,g\right\rangle _{\Omega }=\dint\limits_{\Omega }\overline{f}%
gd\Omega  \label{inp}
\end{equation}

$f,g\in L^{2}\left( \Omega ,Cl_{n}\right) $ is a Hilbert space and therefore
has an orthogonal relationship given by:

\ \ 

\begin{proposition}
In the Hilbert space $L^{2}\left( \Omega ,Cl_{n}\right) $, with respect to
the inner product (\ref{inp})$\ $the orthogonal space $\left( B_{\alpha
}^{2}\left( \Omega ,Cl_{n}\right) \right) ^{\bot }$ of the generalized
Bergman space $B_{\alpha }^{2}\left( \Omega ,Cl_{n}\right) $ is given by :

\begin{equation*}
\left( B_{\alpha }^{2}\left( \Omega ,Cl_{n}\right) \right) ^{\bot }=%
\overline{\Gamma }_{\alpha }\left( W_{0}^{2,1}\left( \Omega ,Cl_{n}\right)
\right)
\end{equation*}

where the Bergman space $B_{\alpha }^{2}\left( \Omega ,Cl_{n}\right) $ is
the set of all Clifford valued square integrable functions which are
annihilated by the spherical Dirac operator $\Gamma _{\alpha }$ over $\Omega 
$ .
\end{proposition}

\begin{proof}
First lets prove that $B_{\alpha }^{2}\left( \Omega ,Cl_{n}\right) \cap 
\overline{\Gamma }_{\alpha }\left( W_{0}^{2,1}\left( \Omega ,Cl_{n}\right)
\right) $ is $\{0\}$, the singleton with only the zero function as the
element.

Indeed, for $f\in B_{\alpha }^{2}\left( \Omega ,Cl_{n}\right) \cap \overline{%
\Gamma }_{\alpha }\left( W_{0}^{2,1}\left( \Omega ,Cl_{n}\right) \right) $,
we have $\Gamma _{\alpha }f=0$ on $\Omega $ and $f=\overline{\Gamma }%
_{\alpha }g$, for $g\in W_{0}^{2,1}\left( \Omega ,Cl_{n}\right) $.

Then $\Gamma _{\alpha }f=\Delta _{\alpha ,0}g=0\Rightarrow g\equiv 0$ on $%
\Omega $. Therefore $f\equiv 0$ on $\Omega $.

\ \ 

Also for $f\epsilon L^{2}\left( \Omega ,Cl_{n}\right) $ we have $f=Pf+Qf$
with $Pf=f-\Gamma _{\alpha }\left( \Delta _{\alpha ,0}^{-1}\Gamma _{\alpha
}f\right) $ and $Qf=\Gamma _{\alpha }\left( \Delta _{\alpha ,0}^{-1}\Gamma
_{\alpha }f\right) $ with $Pf\in B_{\alpha }^{2}\left( \Omega ,Cl_{n}\right) 
$ and $Qf\in \overline{\Gamma }_{\alpha }\left( W_{0}^{2,1}\left( \Omega
,Cl_{n}\right) \right) $,

where $P$ is the Bergman projection and $Q$ is its orthogonal complement.
\end{proof}

\bigskip\ \ \ \ 

As usual, the two orthogonal projections we use in the proof of the above
orthogonality relations are

\begin{equation*}
P:L^{2}\left( \Omega ,Cl_{n}\right) \rightarrow B_{\alpha }^{2}\left( \Omega
,Cl_{n}\right)
\end{equation*}%
which is the Bergman projection and 
\begin{equation*}
Q:L^{2}\left( \Omega ,Cl_{n}\right) \rightarrow \overline{\Gamma }_{\alpha
}\left( W_{0}^{2,1}\left( \Omega ,Cl_{n}\right) \right)
\end{equation*}%
is its orthogonal complement with

$Q=I-P$, and 
\begin{equation*}
PQ=0=QP,P^{2}=P,Q^{2}=Q.
\end{equation*}

\begin{proposition}
For $\phi \in B_{\alpha }^{2}\left( \Omega ,Cl_{n}\right) $ and $\psi \in
\left( B_{\alpha }^{2}\left( \Omega ,Cl_{n}\right) \right) ^{\bot }$, the
squared norm defined by $\mid \mid \mid \bullet \mid \mid \mid :=\parallel
\bullet \parallel _{L^{2}\left( \Omega ,Cl_{n}\right) }^{2}$do satisfy the
following equalities: $\forall n\in 
%TCIMACRO{\U{2115} }%
%BeginExpansion
\mathbb{N}
%EndExpansion
,$

(a)$\qquad \parallel \phi +\psi \parallel _{L^{2}\left( \Omega
,Cl_{n}\right) }^{n}=\left( \parallel \phi \parallel _{L^{2}\left( \Omega
,Cl_{n}\right) }^{2}+\parallel \psi \parallel _{L^{2}\left( \Omega
,Cl_{n}\right) }^{2}\right) ^{\frac{n}{2}}$

(b)$\qquad \mid \mid \mid \phi +\psi \mid \mid \mid ^{n}=\left( \mid \mid
\mid \phi \mid \mid \mid +\mid \mid \mid \psi \mid \mid \mid \right) ^{n}$
\end{proposition}

\begin{proof}
Here the proof can be done using induction on $n:$

Since $\phi \in \left\langle \psi \right\rangle ^{\perp },$ the orthogonal
space of the space $\left\langle \psi \right\rangle $ spanned by $\psi $, we
have $\dint\limits_{\Omega }\overline{\phi }\psi d\Omega
=0=\dint\limits_{\Omega }\overline{\psi }\phi d\Omega $ which implies, 
\begin{equation*}
\parallel \phi +\psi \parallel _{L^{2}\left( \Omega ,Cl_{n}\right)
}^{2}=\left\langle \phi +\psi ,\phi +\psi \right\rangle _{\Omega }
\end{equation*}

\begin{equation*}
=\dint\limits_{\Omega }\overline{\left( \phi +\psi \right) }\left( \phi
+\psi \right) d\Omega =\dint\limits_{\Omega }\overline{\phi }\phi d\Omega
+\dint\limits_{\Omega }\overline{\psi }\psi d\Omega
\end{equation*}

\begin{equation*}
=\parallel \phi \parallel _{L^{2}\left( \Omega ,Cl_{n}\right)
}^{2}+\parallel \psi \parallel _{L^{2}\left( \Omega ,Cl_{n}\right) }^{2}
\end{equation*}%
That is 
\begin{equation*}
\parallel \phi +\psi \parallel _{L^{2}\left( \Omega ,Cl_{n}\right) }=\left(
\parallel \phi \parallel _{L^{2}\left( \Omega ,Cl_{n}\right) }^{2}+\parallel
\psi \parallel _{L^{2}\left( \Omega ,Cl_{n}\right) }^{2}\right) ^{\frac{1}{2}%
}
\end{equation*}%
Therfore the statment is valid for $n=1.$ We assume it is true for $k$, that
is 
\begin{equation*}
\parallel \phi +\psi \parallel _{L^{2}\left( \Omega ,Cl_{n}\right)
}^{k}=\left( \parallel \phi \parallel _{L^{2}\left( \Omega ,Cl_{n}\right)
}^{2}+\parallel \psi \parallel _{L^{2}\left( \Omega ,Cl_{n}\right)
}^{2}\right) ^{\frac{k}{2}}
\end{equation*}%
Then 
\begin{equation*}
\parallel \phi +\psi \parallel _{L^{2}\left( \Omega ,Cl_{n}\right)
}^{k+1}=\parallel \phi +\psi \parallel _{L^{2}\left( \Omega ,Cl_{n}\right)
}^{k}\parallel \phi +\psi \parallel _{L^{2}\left( \Omega ,Cl_{n}\right) }^{1}
\end{equation*}

\begin{eqnarray*}
&=&\left( \parallel \phi \parallel _{L^{2}\left( \Omega ,Cl_{n}\right)
}^{2}+\parallel \psi \parallel _{L^{2}\left( \Omega ,Cl_{n}\right)
}^{2}\right) ^{\frac{k}{2}}\left( \parallel \phi \parallel _{L^{2}\left(
\Omega ,Cl_{n}\right) }^{2}+\parallel \psi \parallel _{L^{2}\left( \Omega
,Cl_{n}\right) }^{2}\right) ^{\frac{1}{2}} \\
&=&\left( \parallel \phi \parallel _{L^{2}\left( \Omega ,Cl_{n}\right)
}^{2}+\parallel \psi \parallel _{L^{2}\left( \Omega ,Cl_{n}\right)
}^{2}\right) ^{\frac{k+1}{2}}
\end{eqnarray*}%
\ which shows the validity of the statement for $k+1$ and that proves the
statement for $\forall n\in 
%TCIMACRO{\U{2115} }%
%BeginExpansion
\mathbb{N}
%EndExpansion
,$ and that proves \textit{(a)}.

\ \ \ 

\textit{(b)} follows easily: 
\begin{eqnarray*}
&\mid &\mid \mid \phi +\psi \mid \mid \mid ^{n}=\left( \parallel \phi +\psi
\parallel _{L^{2}\left( \Omega ,Cl_{n}\right) }^{2}\right) ^{n} \\
&=&\left( \parallel \phi \parallel _{L^{2}\left( \Omega ,Cl_{n}\right)
}^{2}+\parallel \psi \parallel _{L^{2}\left( \Omega ,Cl_{n}\right)
}^{2}\right) ^{n} \\
&=&\left( \mid \mid \mid \phi \mid \mid \mid +\mid \mid \mid \psi \mid \mid
\mid \right) ^{n}
\end{eqnarray*}
\end{proof}

In \cite{lakewryan}, the authors\ \ have a decomposition result for Sobolev
spaces 
\begin{equation*}
W^{p,k-1}\left( \Omega ,Cl_{n}\right) =B^{p,k}\left( \Omega ,Cl_{n}\right)
\dotplus D^{k}\left( \overset{0}{W}^{p,2k-1}\left( \Omega ,Cl_{n}\right)
\right)
\end{equation*}%
where $\dotplus $ is a direct sum and when $p=2$ it is an orthogonal sum
with respect to the inner product (\ref{inp}), with corresponding orthogonal
projections

\bigskip

\begin{equation*}
P^{\left( k\right) }:W^{2,k-1}\left( \Omega ,Cl_{n}\right) \rightarrow
B^{2,k}\left( \Omega ,Cl_{n}\right)
\end{equation*}%
and 
\begin{equation*}
Q^{\left( k\right) }:W^{2,k-1}\left( \Omega ,Cl_{n}\right) \rightarrow
D^{k}\left( \overset{0}{W}^{2,2k-1}\left( \Omega ,Cl_{n}\right) \right)
\end{equation*}

with $Q^{\left( k\right) }=I-P^{\left( k\right) }$\ such that 
\begin{equation*}
P^{\left( k\right) }Q^{\left( k\right) }=Q^{\left( k\right) }P^{\left(
k\right) }=0,\left( P^{\left( k\right) }\right) ^{2}=\left( P^{\left(
k\right) }\right) ,\left( Q^{\left( k\right) }\right) ^{2}=Q^{\left(
k\right) }
\end{equation*}%
and \ $D^{k}=\left( \sum_{i=0}^{n}e_{i}\frac{\partial }{\partial x_{i}}%
\right) ^{k}$, is the $k^{th}$ iterate of the Dirac operator.

\ \ \ 

\begin{proposition}
For $f\in L^{2}\left( \Omega ,Cl_{n}\right) $, and $P$, the Bergman
projection, we have 
\begin{equation*}
\left\langle Pf,f\right\rangle _{\Omega }=\left\langle Pf,Pf\right\rangle
_{\Omega }
\end{equation*}
\end{proposition}

\begin{proof}
Let $f\in L^{2}\left( \Omega ,Cl_{n}\right) $. Then 
\begin{equation*}
f=Pf+Qf
\end{equation*}%
with $Pf\in B_{\alpha }^{2}\left( \Omega ,Cl_{n}\right) $ and $Qf\in 
\overline{\Gamma }_{\alpha }\left( W_{0}^{2,1}\left( \Omega ,Cl_{n}\right)
\right) $.

Therefore,%
\begin{equation*}
\left\langle Pf,Qf\right\rangle _{\Omega }=\left\langle Pf,Pf\right\rangle
_{\Omega }=\dint\limits_{\Omega }\overline{Pf}Qfd\Omega =0
\end{equation*}

which implies 
\begin{equation*}
\dint\limits_{\Omega }\overline{Pf}\left( I-Q\right) fd\Omega
=\dint\limits_{\Omega }\left( \overline{Pf}f-\overline{Pf}Qf\right) d\Omega
=0.
\end{equation*}%
That is, 
\begin{equation*}
\dint\limits_{\Omega }\overline{Pf}fd\Omega =\dint\limits_{\Omega }\overline{%
Pf}Qfd\Omega .
\end{equation*}
Therefore,%
\begin{equation*}
\left\langle Pf,f\right\rangle _{\Omega }=\left\langle Pf,Pf\right\rangle
_{\Omega }.
\end{equation*}
\end{proof}

\begin{proposition}
For $f\in L^{2}\left( \Omega ,Cl_{n}\right) $, $\exists g\in
W_{0}^{2,1}\left( \Omega ,Cl_{n}\right) $ such that 
\begin{equation*}
\Gamma _{\alpha }f=\Gamma _{\alpha }\overline{\Gamma }_{\alpha }g.
\end{equation*}
\end{proposition}

\begin{proof}
Let $f\in L^{2}\left( \Omega ,Cl_{n}\right) $. Then%
\begin{equation*}
f=Pf+Qf
\end{equation*}%
with $Pf\in B^{2}\left( \Omega ,Cl_{n}\right) $ and $Qf$ $\in \overline{%
\Gamma }_{\alpha }\left( W_{0}^{2,1}\left( \Omega ,Cl_{n}\right) \right) $.

Therefore there exists $g\in W_{0}^{2,1}\left( \Omega ,Cl_{n}\right) $ such
that $Qf=\overline{\Gamma }_{\alpha }g$. Then applying $\Gamma _{\alpha }$
on both sides of the equation 
\begin{equation*}
f=Pf+Qf
\end{equation*}%
we have the required result.
\end{proof}

\ \ \ \ 

\begin{remark}
In the case where the Laplacian is factored as%
\begin{equation*}
\Delta =D\overline{D}=\overline{D}D
\end{equation*}%
we could have that $Df=\Delta g$, but in the case of spherical Laplacian the
factorization is a bit different.

For $\Gamma _{\alpha }$ the spherical Dirac operator, the Spherical
Laplacian is factored as%
\begin{equation*}
\Delta _{\alpha }=\Gamma _{\alpha }\Gamma _{\beta }=\Gamma _{\beta }\Gamma
_{\alpha }
\end{equation*}%
where $\alpha +\beta =-n+1$.
\end{remark}

\section{\textbf{Results on the Spherical }$\protect\pi _{\protect\alpha %
,S^{n-1}}-$\textbf{Operator}}

As is done in the case of defining the $\pi $-operator over general domains
in Euclidean spaces, we define $\pi _{\alpha ,S^{n-1}}$\ over domains in the
unit sphere as follows.

\begin{definition}
For $f\in C^{1}\left( \Omega \rightarrow Cl_{n}\right) $ , define 
\begin{equation*}
\pi _{\alpha ,S^{n-1}}(f):=\overline{\Gamma }_{\alpha }T_{\Omega }(f).
\end{equation*}
\end{definition}

In the scale of Sobolev spaces, $\pi _{\alpha ,S^{n-1}}$ is an operator from 
$W^{p,k}\left( \Omega ,Cl_{n}\right) \rightarrow W^{p,k}\left( \Omega
,Cl_{n}\right) $, for $1<p<\infty ,k=0,1,2,...$.

\ \ 

Over domains $\Omega $ in Euclidean spaces $%
%TCIMACRO{\U{211d} }%
%BeginExpansion
\mathbb{R}
%EndExpansion
^{n}\left( \text{ or }%
%TCIMACRO{\U{2102} }%
%BeginExpansion
\mathbb{C}
%EndExpansion
^{n}\right) $, this operator \ has the following integral representation:

For $n=1$: 
\begin{equation*}
\pi _{\Omega }f(w)=\int_{\Omega }\Psi (z-w)f(z)dz
\end{equation*}%
where%
\begin{equation*}
\Psi (z-w)=\frac{-1}{\pi \left( z-w\right) ^{2}}
\end{equation*}

and for $n>1$, we have a representation given by :

\begin{equation*}
\pi _{\Omega }f(x)=\int_{\Omega }-\frac{n+\left( n+2\right) \frac{\overline{%
\left( y-x\right) }^{2}}{\mid y-x\mid ^{2}}}{\omega \mid y-x\mid ^{n+2}}%
f\left( y\right) d\Omega _{y}+\frac{-n}{n+2}f\left( x\right)
\end{equation*}

which in both cases the $\pi _{\Omega }-$operator is a hyper singular
integral operator of C-Z type.

\ \ \ 

Note also that from the above general formula of the $\pi _{\Omega }-$
operator, taking $n=0$, we have the well known $\pi _{\Omega }-$operator
(given above) in the usual $1D$\ complex space \ $%
%TCIMACRO{\U{2102} }%
%BeginExpansion
\mathbb{C}
%EndExpansion
^{1}\simeq Cl_{0,1},$ where $\sqrt{-1}=i\sim e_{1}.$

\ \ 

Also in a generalized setting, Dejenie A. Lakew and John Ryan in \cite%
{lakryan} study the $\pi -$operator over real, compact $\left( n+1\right) -$
manifolds in $%
%TCIMACRO{\U{2102} }%
%BeginExpansion
\mathbb{C}
%EndExpansion
^{n+1}$ which are called Domain Manifolds and produced an integral
representation of $\pi $ over such manifolds\ given by: 
\begin{equation*}
\pi _{\Omega }f(w)=\dint\limits_{\Omega }\overline{D}_{\gamma }\Psi ^{\Gamma
}\left( w-v\right) f\left( v\right) d\Omega _{v}+\frac{-n+1}{n+1}f\left(
w\right)
\end{equation*}%
where $D_{\gamma }$ is a nonhomogeneous Dirac like operator defined by $%
D_{\gamma }=\dsum\limits_{j=0}^{n}e_{j}\left( \frac{\partial }{\partial x_{j}%
}-\gamma _{j}\right) $, $\Psi ^{\Gamma }$ is a fundamental solution for $%
D_{\gamma }$ \ and $\Omega $ is a domain manifold in $%
%TCIMACRO{\U{2102} }%
%BeginExpansion
\mathbb{C}
%EndExpansion
^{n+1}.$

\begin{definition}
We define the Clifford conjugate of $\pi _{\alpha ,S^{n-1}}$ as 
\begin{equation*}
\overline{\pi }_{\alpha ,S^{n-1}}:=\Gamma _{\alpha }\overline{T}_{\Omega }
\end{equation*}%
where 
\begin{equation*}
\overline{T}_{\Omega }\left( f\right) (x)=\int_{\Omega }\overline{\Psi
_{\alpha }^{n}}\left( y,x\right) f(y)d\Omega _{y}\text{.}
\end{equation*}
\end{definition}

\begin{proposition}
(Classical Analogous Results:\cite{gurkah},\cite{gurmal})

On the Sobolev space $W^{p,k}\left( \Omega ,Cl_{n}\right) ,$ where $%
1<p<\infty $, $k=0,1,2,...$, we have :

\begin{enumerate}
\item $\qquad \Gamma _{\alpha }\pi _{\alpha ,S^{n-1}}=\overline{\Gamma }%
_{\alpha }$

\item $\qquad \pi _{\alpha ,S^{n-1}}\Gamma _{\alpha }=\overline{\Gamma }%
_{\alpha }\left( I-F_{\partial \Omega }\right) $

\item $\qquad F_{\partial \Omega }\pi _{\alpha ,S^{n-1}}=\pi _{\alpha
,S^{n-1}}-T_{\Omega }\overline{\Gamma }_{\alpha }$

\item $\qquad \Gamma _{\alpha }\pi _{\alpha ,S^{n-1}}-\pi _{\alpha ,S^{n-1}}=%
\overline{\Gamma }_{\alpha }F_{\partial \Omega }$
\end{enumerate}
\end{proposition}

\begin{corollary}
From the above proposition we see that 
\begin{equation*}
\pi _{\alpha ,S^{n-1}}:\overline{B}_{\alpha }^{2}\left( \Omega
,Cl_{n}\right) \rightarrow B_{\alpha }^{2}\left( \Omega ,Cl_{n}\right)
\end{equation*}%
where 
\begin{equation*}
\overline{B}_{\alpha }^{2}\left( \Omega ,Cl_{n}\right) =L^{2}\left( \Omega
,Cl_{n}\right) \cap \ker \overline{\Gamma }_{\alpha }\left( \Omega
,Cl_{n}\right) \text{ \ and \ }B_{\alpha }^{2}\left( \Omega ,Cl_{n}\right)
\end{equation*}%
is the Bergman space $L^{2}\left( \Omega ,Cl_{n}\right) \cap \ker \Gamma
_{\alpha }\left( \Omega ,Cl_{n}\right) $
\end{corollary}

\begin{corollary}
Also we have 
\begin{equation*}
\pi _{\alpha ,S^{n-1}}:\Gamma _{\alpha }\left( W_{0}^{2,1}\left( \Omega
,Cl_{n}\right) \right) \rightarrow \overline{\Gamma }_{\alpha }\left(
W_{0}^{2,1}\left( \Omega ,Cl_{n}\right) \right) .
\end{equation*}
\end{corollary}

\begin{proposition}
Let $g\in W_{0}^{p,k}\left( \Omega ,Cl_{n}\right) ,1<p<\infty ,k=0,1,2,...$.
Then

\begin{enumerate}
\item $\qquad \pi _{\alpha ,S^{n-1}}\Gamma _{\alpha }g=\overline{\Gamma }%
_{\alpha }g$

\item \qquad $\Gamma _{\alpha }\pi _{\alpha ,S^{n-1}}=\pi _{\alpha ,S^{n-1}}$

\item \qquad $\pi _{\alpha ,S^{n-1}}\Gamma _{\alpha }g=\Gamma _{\alpha }\pi
_{\alpha ,S^{n-1}}g=\pi _{\alpha ,S^{n-1}}g$
\end{enumerate}
\end{proposition}

\begin{proof}
If $g\epsilon W^{p,k}\left( \Omega ,Cl_{n}\right) $ is compactly supported
over $\Omega $, then the boundary integral of $g$ is zero. That is $%
F_{\partial \Omega }g=0$, and therefore, from%
\begin{equation*}
\pi _{\alpha ,S^{n-1}}\Gamma _{\alpha }=\overline{\Gamma }_{\alpha }\left(
I-F_{\partial \Omega }\right)
\end{equation*}%
we have 
\begin{equation*}
\pi _{\alpha ,S^{n-1}}\Gamma _{\alpha }=\overline{\Gamma }_{\alpha }
\end{equation*}%
since the $\partial -$integral is zero, which proves (1)

Also from 
\begin{equation*}
\Gamma _{\alpha }\pi _{\alpha ,S^{n-1}}g-\pi _{\alpha ,S^{n-1}}g=\overline{%
\Gamma }_{\alpha }F_{\partial \Omega }\left( tr_{\partial \Omega }g\right)
\end{equation*}
as $g$ is compactly supported over the domain, its boundary integral over
the domain is zero, i.e. $F_{\partial \Omega }\left( tr_{\partial \Omega
}g\right) =0$.

Thus we have%
\begin{equation*}
\Gamma _{\alpha }\pi _{\alpha ,S^{n-1}}g-\pi _{\alpha ,S^{n-1}}g=\overline{%
\Gamma }_{\alpha }F_{\partial \Omega }\left( tr_{\partial \Omega }g\right) =0
\end{equation*}%
which implies 
\begin{equation*}
\Gamma _{\alpha }\pi _{\alpha ,S^{n-1}}g=\pi _{\alpha ,S^{n-1}}g
\end{equation*}
which proves (2).

Finally, 
\begin{equation*}
\Gamma _{\alpha }\pi _{\alpha ,S^{n-1}}g=\overline{\Gamma }_{\alpha }g
\end{equation*}%
on the Sobolev space $W^{p,k}\left( \Omega ,Cl_{n}\right) $ and in
particular when $tr_{\partial \Omega }g=0$, we have 
\begin{equation*}
\pi _{\alpha ,S^{n-1}}\Gamma _{\alpha }g=\overline{\Gamma }_{\alpha }g
\end{equation*}%
adjoining this with result (2) we prove (3).
\end{proof}

\ \ \ 

\begin{remark}
On the space of functions with compact support, the action of the spherical
Dirac operator from the right and from the left on $\pi _{\alpha ,S^{n-1}}$
is irrelevant.
\end{remark}

\begin{corollary}
If $f$ is a smooth Clifford valued function which has a compact support over 
$\Omega $, then $\pi _{\alpha ,S^{n-1}}$ and $\Gamma _{\alpha }$ commute at $%
f$ and furthermore, their product \ at such a function is the Clifford
conjugate $\overline{\Gamma }_{\alpha }$\ of $\Gamma _{\alpha }$.
\end{corollary}

\begin{proposition}
In the Sobolev space $W^{2,1}\left( \Omega ,Cl_{n}\right) $, 
\begin{equation*}
f\in \ker \overline{\Gamma }_{\alpha }\Rightarrow \text{ }\pi _{\alpha
,S^{n-1}}f\in \ker \Gamma _{\alpha }
\end{equation*}%
That is 
\begin{equation*}
\pi _{\alpha ,S^{n-1}}:W^{2,1}\left( \Omega \right) \cap \ker \overline{%
\Gamma }_{\alpha }\rightarrow W^{2,1}\left( \Omega \right) \cap \ker \Gamma
_{\alpha }.
\end{equation*}
\end{proposition}

\begin{proof}
Let $f\in W^{2,1}\left( \Omega ,Cl_{n}\right) \cap \ker \overline{\Gamma }%
_{\alpha }$. Then $\pi _{\alpha ,S^{n-1}}f\in W^{2,1}\left( \Omega
,Cl_{n}\right) $ as $\pi _{\alpha ,S^{n-1}}$ preserves regularity, and $%
\overline{\Gamma }_{\alpha }f=0$. But from the relation 
\begin{equation*}
\overline{\Gamma }_{\alpha }=\Gamma _{\alpha }\pi _{\alpha ,S^{n-1}}
\end{equation*}%
we get 
\begin{equation*}
\Gamma _{\alpha }\pi _{\alpha ,S^{n-1}}f=0
\end{equation*}%
which is the required result. From this result, one can see that the $\pi
_{\alpha }$ operator preserves monogenicity or hypercomplex regularity of
functions in the sense : 
\begin{equation*}
\pi _{\alpha ,S^{n-1}}:\overline{B}_{\alpha }^{2}\left( \Omega
,Cl_{n}\right) \rightarrow B_{\alpha }^{2}\left( \Omega ,Cl_{n}\right)
\end{equation*}%
where 
\begin{equation*}
\overline{B}_{\alpha }^{2}\left( \Omega ,Cl_{n}\right) =L^{2}\left( \Omega
,Cl_{n}\right) \cap \ker \overline{\Gamma }_{\alpha }\text{ and \ }B_{\alpha
}^{2}\left( \Omega ,Cl_{n}\right)
\end{equation*}%
is the Bergman space mentioned above.
\end{proof}

\ \ \ \ \ 

\begin{proposition}
Let%
\begin{equation*}
\Lambda _{0,j}:=e_{0}e_{j}\det \left( 
\begin{array}{cc}
\omega _{0} & \omega _{j} \\ 
\frac{\partial }{\partial \omega _{0}} & \frac{\partial }{\partial \omega
_{j}}%
\end{array}%
\right) ,j=1,...,n.
\end{equation*}

If $\pi _{\alpha ,S^{n-1}}$ fixes $f\in L^{2}\left( \Omega ,Cl_{n}\right) $,
then $f$ satisfies the equation 
\begin{equation*}
f=\left( \dsum\limits_{0<j}\Lambda _{0,j}+\alpha \right) T_{\Omega }f.
\end{equation*}
\end{proposition}

\begin{proof}
First we rewrite the conjugate of the spherical Dirac operator $\overline{%
\Gamma }_{\alpha }$ as 
\begin{equation*}
\overline{\Gamma }_{\alpha }=2\left( \dsum\limits_{0<j}\underset{\Lambda
_{0,j}}{\underbrace{e_{0}e_{j}\left( \omega _{0}\frac{\partial }{\partial
\omega _{j}}-\omega _{j}\frac{\partial }{\partial \omega _{0}}\right) }}%
+\alpha \right) -\Gamma _{\alpha }.
\end{equation*}%
Then 
\begin{equation*}
\pi _{\alpha ,S^{n-1}}f=f
\end{equation*}%
implies%
\begin{equation*}
\overline{\Gamma }_{\alpha }T_{\Omega }f=f.
\end{equation*}%
Using the expression for $\overline{\Gamma }_{\alpha }$ in terms of $\Gamma
_{\alpha }$ given above we get the desired result.
\end{proof}

\ \ \ \ \ \ \ \ 

\begin{remark}
From the above result and the orthogonal decomposition of the Hilbert space,
we can see that the spherical $\pi _{\alpha ,S^{n-1}}$ has the mapping
property:

\begin{equation*}
\pi _{\alpha ,S^{n-1}}:\overline{\Gamma }_{\alpha }\left( W_{0}^{2,1}\left(
\Omega ,Cl_{n}\right) \right) \rightarrow \Gamma _{\alpha }\left(
W_{0}^{2,1}\left( \Omega ,Cl_{n}\right) \right)
\end{equation*}

pictorially, we describe the above mapping properties as(for further
studies, see \cite{gurkah},\cite{gurmal}):

\begin{equation*}
L^{2}\left( \Omega ,Cl_{n}\right) =\left. 
\begin{array}{c}
\underset{\downarrow }{\underbrace{B_{\alpha }^{2}\left( \Omega
,Cl_{n}\right) }}\oplus \underset{\downarrow }{\underbrace{\overline{\Gamma }%
_{\alpha }\left( W_{0}^{2,1}\left( \Omega ,Cl_{n}\right) \right) }} \\ 
\overline{B}_{\alpha }^{2}\left( \Omega ,Cl_{n}\right) \oplus \Gamma
_{\alpha }\left( W_{0}^{2,1}\left( \Omega ,Cl_{n}\right) \right)%
\end{array}%
\right\} \downarrow :\pi _{\alpha ,S^{n-1}}
\end{equation*}
\end{remark}

\begin{proposition}
Let $p\in \left( 1,\infty \right) ,k\in 
%TCIMACRO{\U{2124} }%
%BeginExpansion
\mathbb{Z}
%EndExpansion
\cup \left\{ 0\right\} $ and $f\in W^{p,k}\left( \Omega ,Cl_{n}\right) $.
Then 
\begin{equation*}
\overline{\pi }_{\alpha ,S^{n-1}}\pi _{\alpha ,S^{n-1}}f+\Gamma _{\alpha }%
\overline{F}_{\partial \Omega }T_{\Omega }f=f.
\end{equation*}
\end{proposition}

\begin{corollary}
By taking the complexified Clifford conjugate of the above equation we get 
\begin{equation*}
\pi _{\alpha ,S^{n-1}}\overline{\pi }_{\alpha ,S^{n-1}}f+\overline{\Gamma }%
_{\alpha }F_{\partial \Omega }\overline{T}_{\Omega }f=f
\end{equation*}
\end{corollary}

\begin{corollary}
The $\pi _{\alpha ,S^{n}}$ operator is left invertible on the space $\Gamma
_{\alpha }\left( W_{0}^{2,1}\left( \Omega ,Cl_{n}\right) \right) $ with left
inverse of $\overline{\pi }_{\alpha ,S^{n-1}}$.
\end{corollary}

\begin{proof}
Let $f\in \Gamma _{\alpha }\left( W_{0}^{2,1}\left( \Omega ,Cl_{n}\right)
\right) $. Then there exists a function $g\in W_{0}^{2,1}\left( \Omega
,Cl_{n}\right) $ such that $f=\Gamma _{\alpha }g$ with $tr_{\partial \Omega
}g=0$.

From the Borel-Pompeiu formula, we have 
\begin{equation*}
g=T_{\Omega }\Gamma _{\alpha }g=T_{\Omega }f
\end{equation*}%
and this implies 
\begin{equation*}
tr_{\partial \Omega }T_{\Omega }f=0
\end{equation*}%
and therefore, 
\begin{equation*}
F_{\partial \Omega }T_{\Omega }f=F_{\partial \Omega }\left( tr_{\partial
\Omega }T_{\Omega }f\right) =0
\end{equation*}%
which yields the result.
\end{proof}

\ \ 

\begin{remark}
A similar argument yields that $\pi _{\alpha ,S^{n-1}}$ is right invertible
on $\overline{\Gamma }_{\alpha }\left( W_{0}^{2,1}\left( \Omega
,Cl_{n}\right) \right) $ with right inverse of $\overline{\pi }_{\alpha
,S^{n-1}}$.
\end{remark}

Denote by $\Xi $ the overlap of the function spaces $\overline{\Gamma }%
_{\alpha }\left( W_{0}^{2,1}\left( \Omega ,Cl_{n}\right) \right) \cap \Gamma
_{\alpha }\left( W_{0}^{2,1}\left( \Omega ,Cl_{n}\right) \right) $.

\begin{lemma}
$\Xi $ is non empty.
\end{lemma}

\begin{proof}
Here, to show that the set $\Xi $ is non-empty, we need to consider the
Sobolev space$\left( \Gamma _{\alpha }\overline{\Gamma }_{\alpha }\right)
\left( W_{0}^{2,2}\left( \Omega ,Cl_{n}\right) \right) $ where $\overline{%
\Gamma }_{\alpha }$ is the Clifford conjugate of the spherical Dirac
operator of order $\alpha $. Let $f\in \Gamma _{\alpha }\overline{\Gamma }%
_{\alpha }\left( W_{0}^{2,2}\left( \Omega ,Cl_{n}\right) \right) $. Then $%
\exists g\in \Gamma _{\alpha }\left( W_{0}^{2,2}\left( \Omega ,Cl_{n}\right)
\right) $ and $h\in \overline{\Gamma }_{\alpha }\left( W_{0}^{2,2}\left(
\Omega ,Cl_{n}\right) \right) $ $\ni f=\Gamma _{\alpha }g=\overline{\Gamma }%
_{\alpha }h$.

That is $f\in \Xi $. In fact the argument shows that $\Gamma _{\alpha }%
\overline{\Gamma }_{\alpha }\left( W_{0}^{2,2}\left( \Omega ,Cl_{n}\right)
\right) \subseteq \Xi .$
\end{proof}

\ \ \ 

\begin{proposition}
On the space $\Xi $, $\pi _{\alpha ,S^{n-1}}$ is invertible with inverse of $%
\overline{\pi }_{\alpha ,S^{n-1}}.$
\end{proposition}

When we consider global functions over the sphere we may have better results
on invertibility of $\pi _{\alpha ,S^{n-1}}$.

\begin{proposition}
On the space $C_{0}^{\infty }\left( S^{n-1},Cl_{n}\right) $, we have 
\begin{equation*}
\overline{\pi }_{\alpha ,S^{n-1}}\pi _{\alpha ,S^{n-1}}=\pi _{\alpha
,S^{n-1}}\overline{\pi }_{\alpha ,S^{n-1}}
\end{equation*}
\end{proposition}

From denseness arguments, and boundedness of the $\pi _{\alpha ,S^{n-1}}$ on 
$L^{2}\left( \Omega ,Cl_{n}\right) ,$ the above result can be done over a
larger domain as:

\begin{corollary}
On the space $L^{2}\left( S^{n-1},Cl_{n}\right) $, we have 
\begin{equation*}
\overline{\pi }_{\alpha ,S^{n-1}}\pi _{\alpha ,S^{n-1}}=\pi _{\alpha
,S^{n-1}}\overline{\pi }_{\alpha ,S^{n-1}}
\end{equation*}
\end{corollary}

With respect to the Clifford valued inner product given by $\left( \text{\ref%
{inp}}\right) $ on the Hilbert space $L^{2}\left( \Omega ,Cl_{n}\right) $,

we take $-\overline{T}_{\Omega }$ as the adjoint $T_{\Omega }^{\ast }$ of $%
T_{\Omega }$ and $-\overline{\Gamma }_{\alpha }$ as $\Gamma _{\alpha }^{\ast
},$ adjoint of $\Gamma _{\alpha }$.

\ \ 

Therefore for $f,g\in W_{0}^{2,k}\left( \Omega ,Cl_{n}\right) $, we have

\begin{eqnarray*}
\left\langle \Gamma _{\alpha }f,g\right\rangle _{\Omega }
&=&\dint\limits_{\Omega }\overline{\Gamma _{\alpha }f}gd\Omega \\
&=&\dint\limits_{\Omega }\overline{\Gamma }_{\alpha }\overline{f}gd\Omega \\
&=&-\dint\limits_{\Omega }\overline{f}\overline{\Gamma }_{\alpha }gd\Omega \\
&=&-\left\langle f,\overline{\Gamma }_{\alpha }g\right\rangle _{\Omega }
\end{eqnarray*}

and 
\begin{eqnarray*}
\ \left\langle T_{\Omega }f,g\right\rangle _{\Omega }
&=&\dint\limits_{\Omega }\overline{T_{\Omega }f}gd_{\Omega _{\upsilon }} \\
&=&\dint\limits_{\Omega }\overline{T}_{\Omega }\overline{f}gd_{\Omega
_{\upsilon }}
\end{eqnarray*}

\begin{equation*}
=\dint\limits_{\Omega }\left( \int_{\Omega }\overline{\Psi _{\alpha }^{n}}%
\left( \omega ,\upsilon \right) \overline{f}(\omega )d\Omega _{\omega
}\right) g\left( \upsilon \right) d\Omega _{\upsilon }
\end{equation*}

\begin{eqnarray*}
&=&\dint\limits_{\Omega \times \Omega }\overline{\Psi _{\alpha }^{n}}\left(
\omega ,\upsilon \right) \overline{f}(\omega )g\left( \upsilon \right)
d\Omega _{\omega }d\Omega _{\upsilon } \\
&=&-\dint\limits_{\Omega }\overline{f}\left( \omega \right) d\Omega _{\omega
}\dint\limits_{\Omega }\overline{\Psi _{\alpha }^{n}}\left( \omega ,\upsilon
\right) g\left( \upsilon \right) d\Omega _{\upsilon }
\end{eqnarray*}

\begin{eqnarray*}
&=&-\dint\limits_{\Omega }\overline{f}\left( \omega \right) \left( \overline{%
T}_{\Omega }g\left( \upsilon \right) d\Omega _{\upsilon }\right) d\Omega
_{\omega } \\
&=&-\left\langle f,\overline{T}_{\Omega }g\right\rangle _{\Omega }.
\end{eqnarray*}

\begin{lemma}
$\pi _{\alpha }^{\ast }=\overline{T}_{\Omega }\Gamma _{\alpha }$
\end{lemma}

\begin{proof}
\begin{eqnarray*}
\left\langle \pi _{\alpha ,S^{n-1}}f,g\right\rangle &=&\left\langle 
\overline{\Gamma }_{\alpha }T_{\Omega }f,g\right\rangle \\
&=&-\left\langle T_{\Omega }f,\Gamma _{\alpha }g\right\rangle \\
&=&\left\langle f,\overline{T}_{\Omega }\Gamma _{\alpha }g\right\rangle \\
&=&\left\langle f,\pi _{\alpha ,S^{n-1}}^{\ast }\right\rangle
\end{eqnarray*}%
From this we can see that the adjoint $\pi _{\alpha ,S^{n-1}}^{\ast }$of the 
$\pi _{\alpha ,S^{n-1}}$ operator is $\overline{T}_{\Omega }\Gamma _{\alpha
} $
\end{proof}

\begin{proposition}
On $W_{0}^{2,k}\left( \Omega ,Cl_{n}\right) $,$k=0,1,2,...$, we have 
\begin{equation*}
\pi _{\alpha ,S^{n-1}}^{\ast }\pi _{\alpha ,S^{n-1}}=I_{S^{n-1}}.
\end{equation*}
\end{proposition}

\begin{proof}
For $f\in W_{0}^{2,k}\left( \Omega ,Cl_{n}\right) $, 
\begin{equation*}
\pi _{\alpha ,S^{n-1}}^{\ast }\pi _{\alpha ,S^{n-1}}f=\overline{T}_{\Omega
}\Gamma _{\alpha }\pi _{\alpha ,S^{n-1}}f
\end{equation*}

\begin{eqnarray*}
&=&\overline{T}_{\Omega }\Gamma _{\alpha }\overline{\Gamma }_{\alpha
}T_{\Omega }f \\
&=&\overline{T}_{\Omega }\overline{\Gamma }_{\alpha }\Gamma _{\alpha
}T_{\Omega }f \\
&=&\overline{I}_{\Omega }I_{\Omega }f=f.
\end{eqnarray*}%
This is because for a function $f\in W^{2,k}\left( \Omega ,Cl_{n}\right) $
whose trace is zero over the boundary, $\Gamma _{\alpha }$ is both the right
and left inverse of the $T_{\Omega }$ operator.
\end{proof}

\begin{corollary}
On the Hilbert space $L^{2}\left( \Omega ,Cl_{n}\right) $, 
\begin{equation*}
\pi _{\alpha ,S^{n-1}}^{\ast }=\overline{\pi }_{\alpha ,S^{n-1}}.
\end{equation*}
\end{corollary}

From the above corollary we get that for $f,g\in L^{2}\left( \Omega
,Cl_{n}\right) $,

\begin{eqnarray*}
\left\langle \pi _{\alpha ,S^{n-1}}f,\pi _{\alpha ,S^{n-1}}g\right\rangle
_{\Omega } &=&\left\langle f,\pi _{\alpha ,S^{n-1}}^{\ast }\pi _{\alpha
,S^{n-1}}g\right\rangle _{\Omega } \\
&=&\left\langle f,g\right\rangle _{\Omega }.
\end{eqnarray*}%
By taking $f=g$, we have an isometry property for the $\pi _{\alpha
,S^{n-1}} $ operator over $L^{2}\left( \Omega ,Cl_{n}\right) :$

\begin{proposition}
$\parallel \pi _{\alpha ,S^{n-1}}f\parallel =\parallel f\parallel $ for $%
f\in W^{2,0}\left( \Omega ,Cl_{n}\right) $, i.e., $\pi _{\alpha ,S^{n-1}}$
is norm preserving over the Hilbert space.
\end{proposition}

\begin{proof}
For $f\in W^{2,0}\left( \Omega ,Cl_{n}\right) $,%
\begin{equation*}
\parallel \pi _{\alpha ,S^{n-1}}f\parallel =\left[ \left\langle \pi _{\alpha
,S^{n-1}}f,\pi _{\alpha ,S^{n-1}}f\right\rangle _{\Omega }\right] _{0}
\end{equation*}

\begin{eqnarray*}
&=&\left[ \left\langle f,\pi _{\alpha ,S^{n-1}}^{\ast }\pi _{\alpha
,S^{n-1}}f\right\rangle _{\Omega }\right] _{0} \\
&=&\left[ \left\langle f,f\right\rangle _{\Omega }\right] _{0}=\parallel
f\parallel
\end{eqnarray*}
\end{proof}

In the following proposition we identify functions in the Hilbert space
which are fixed by the spherical $\pi _{\alpha ,S^{n-1}}$ operator?

\begin{proposition}
(Fixed Points of $\pi _{\alpha ,S^{n-1}}$) Let $f\in L^{2}\left( \Omega
,Cl_{n}\right) $. If 
\begin{equation*}
\pi _{\alpha ,S^{n-1}}f=f
\end{equation*}%
then%
\begin{equation*}
f=\left( \dsum\limits_{0<j}e_{0}e_{j}\left( \omega _{0}\frac{\partial }{%
\partial \omega _{j}}-\omega _{j}\frac{\partial }{\partial \omega _{0}}%
\right) +\alpha \right) T_{\Omega }f.
\end{equation*}
\end{proposition}

\begin{proof}
First%
\begin{eqnarray*}
\Gamma _{\alpha } &=&\Gamma _{\omega }+\alpha \\
&=&-\dsum\limits_{i<j}e_{ij}\left( \omega _{i}\frac{\partial }{\partial
\omega _{j}}-\omega _{j}\frac{\partial }{\partial \omega _{i}}\right) +\alpha
\end{eqnarray*}
Then

\begin{equation*}
\overline{\Gamma }_{\alpha }=-2\left( \dsum\limits_{0<j}e_{0}e_{j}\left(
\omega _{0}\frac{\partial }{\partial \omega _{j}}-\omega _{j}\frac{\partial 
}{\partial \omega 0}\right) +\alpha \right) -\Gamma _{\alpha }.
\end{equation*}
This gives us

\begin{equation*}
\pi _{\alpha ,S^{n-1}}f=\left( -2\left( \dsum\limits_{0<j}e_{0}e_{j}\left(
\omega _{0}\frac{\partial }{\partial \omega _{j}}-\omega _{j}\frac{\partial 
}{\partial \omega 0}\right) +\alpha \right) -\Gamma _{\alpha }\right)
T_{\Omega }f
\end{equation*}

simplifying this and equating the result to $f$, we get the result.
\end{proof}

\ \ \ \ \ 

\begin{remark}
The $\pi _{\alpha ,S^{n-1}}$ operator is a bounded, and isometric operator
which preserves regularity over Sobolev spaces and yet different from the
Identity operator .
\end{remark}

\begin{remark}
For a global function $g\in L^{2}\left( S^{n-1},Cl_{n}\right) $, or in $%
W_{0}^{2,k}\left( \Omega ,Cl_{n}\right) $, the equation 
\begin{equation*}
\pi _{\alpha ,S^{n-1}}f=g
\end{equation*}%
has a solution in the respective space given by 
\begin{equation*}
f=\pi _{\alpha ,S^{n-1}}^{\ast }g
\end{equation*}%
where $\pi _{\alpha ,S^{n-1}}^{\ast }$ is the adjoint of the $\pi _{\alpha
,S^{n-1}}$ operator.
\end{remark}

\section{\textbf{The Spherical Clifford Beltrami Equation}}

For measurable functions $f,q:\Omega \subseteq 
%TCIMACRO{\U{2102} }%
%BeginExpansion
\mathbb{C}
%EndExpansion
$ with $\parallel q\parallel <1$, the classical Beltrami equation given by%
\begin{equation*}
f_{z}-qf_{\overline{z}}=0
\end{equation*}%
has been studied by many authors. The equation has also its version in
higher dimensions in the real Clifford algebra $Cl_{n}\left( 
%TCIMACRO{\U{211d} }%
%BeginExpansion
\mathbb{R}
%EndExpansion
\right) $ or in the complexified Clifford algebra $Cl_{n}\left( 
%TCIMACRO{\U{2102} }%
%BeginExpansion
\mathbb{C}
%EndExpansion
\right) $ or over domain manifolds in $%
%TCIMACRO{\U{2102} }%
%BeginExpansion
\mathbb{C}
%EndExpansion
^{n}$.

\ \ 

In \cite{lakryan}, the authors study the Beltrami equation over $%
%TCIMACRO{\U{2102} }%
%BeginExpansion
\mathbb{C}
%EndExpansion
^{n+1}$ via real, compact, $(n+1)-$manifolds in $%
%TCIMACRO{\U{2102} }%
%BeginExpansion
\mathbb{C}
%EndExpansion
^{n+1}$. This is possible by introducing an intrinsic Dirac operator
specific to each domain manifold.

\ \ 

In this paper we extend our study of the $\pi $-operator over spherical
domains and once again consider the Beltrami equation here.

\begin{definition}
Let $\Omega $ be a smooth domain in $S^{n-1}$ and let $q:\Omega \rightarrow
Cl_{n}$ be a measurable function.

Then for $f\in W^{2,1}\left( \Omega ,Cl_{n}\right) $, the spherical Clifford
Beltrami equation is given by

\begin{equation}
\Gamma _{\alpha }f-q\overline{\Gamma }_{\alpha }f=0  \label{be1}
\end{equation}
\end{definition}

In order to study this Beltrami equation, lets consider an integral equation
given by

\begin{equation*}
f=T_{\Omega }h+\phi
\end{equation*}%
where $\phi $ is in the $\ker \Gamma _{\alpha }\left( \Omega \right) $ and $%
h=\Gamma _{\alpha }f$ .Then applying $\overline{\Gamma }_{\alpha }$ on both
sides of the integral equation we get%
\begin{equation*}
\overline{\Gamma }_{\alpha }f=\overline{\Gamma }_{\alpha }T_{\Omega }h+%
\overline{\Gamma }_{\alpha }\phi =\pi _{\alpha ,S^{n-1}}h+\widetilde{\phi }
\end{equation*}%
with $\widetilde{\phi }=\overline{\Gamma }_{\alpha }\phi $ and solving for $%
h,$ we get

\begin{equation}
h=\Gamma _{\alpha }f=q\left( \pi _{\alpha ,S^{n-1}}h+\widetilde{\phi }\right)
\label{be2}
\end{equation}

We now consider the two equations $\left( \text{\ref{be1}}\right) $ and $%
\left( \text{\ref{be2}}\right) .$The solvability of one is the solvability
of the other.

To study the solvability of $\left( \text{\ref{be2}}\right) $, we consider
the mapping : 
\begin{equation*}
h\mapsto q\pi _{\alpha ,S^{n-1}}h\text{, \ with}\parallel q\parallel <1
\end{equation*}%
\ 

\ which is a contraction map and hence it has a fixed point which is going
to be a solution to $\left( \text{\ref{be2}}\right) $. Therefore the
Beltrami equation has a solution.

\end{document}